\documentclass[11pt, oneside]{article}   	% use "amsart" instead of "article" for AMSLaTeX format
\usepackage{geometry}                		% See geometry.pdf to learn the layout options. There are lots.
\usepackage{graphicx}				% Use pdf, png, jpg, or eps§ with pdflatex; use eps in DVI mode
\usepackage{amssymb}
\usepackage{amsmath}
\usepackage{graphicx}				% Use pdf, png, jpg, or eps§ with pdflatex; use eps in DVI mode
\usepackage{amssymb}
\usepackage{amsmath}
\usepackage{amsthm}
\usepackage{harpoon}
\usepackage{accents}

\usepackage{tikz}  
\usepackage{float}
\restylefloat{table}
\usepackage{mathrsfs}
\usepackage{verbatim}

\usetikzlibrary{positioning,chains,fit,shapes,calc}  

\newtheorem{theorem}{Theorem}
\newtheorem{lemma}{Lemma}
\newtheorem{definition}{Definition}

\newtheorem{corollary}{Corollary}

\title{On properties of the Taylor series coefficients of the Riemann xi function at $s=\frac{1}{2}$}
\author{Mario DeFranco}
\begin{document}
\maketitle{}
\begin{abstract}
We prove some properties about the non-zero Taylor series coefficients $a_k$ of the Riemann xi function $\xi(s)$ at $s=\frac{1}{2}$. In particular, we present integral formulas that evaluate $a_k$ whose integrands involve a Gaussian function and a function we call $L(x;k)$. We use these formulas to show that $a_k$ is positive. We also define a sequence of polynomials $p(x;n)$ which arise naturally from the integral formulas and use them to prove that the coefficients $a_k$ are decreasing. 
\end{abstract}

\section{Introduction}
The Riemann xi function $\xi(s)$ is an entire function that has a significant role in number theory: it is the ``completed" form of the Riemann zeta function $\zeta(s)$
\[
\xi(s) = -s(1-s)\pi^{-\frac{s}{2}}\Gamma(\frac{s}{2})\zeta(s)
\] 
where 
\[
\zeta(s) = \sum_{n=1}^\infty \frac{1}{n^s}
\]
and $\Gamma(s)$ is the Gamma function
\[
\Gamma(s) = \int_0^\infty e^{-x}x^{s-1}\, dx.
\]
This completion has many desirable properties which $\zeta(s)$ by itself does not posses; there is the functional equation 
\[
\xi(s) = \xi(1-s)
\]
and the fact that, for real $t$, $\xi(\frac{1}{2}+it)$ is real and even. These properties are manifest in the defining integral formula \eqref{xi integral} for $\xi(s)$. Furthermore, it was conjectured by B. Riemann \cite{Riemann} that all the zeros of $\xi(s)$ lie on the line $s=\frac{1}{2}+i t$. Thus it is natural to investigate the Taylor series coefficients of $\xi(s)$ at $s=\frac{1}{2}$. 

We prove some properties about the non-zero Taylor series coefficients $a_k$ (see the definition of the $a_k$ below). In particular, we present an integral formula that evaluates $a_k$ whose integrand is a Gaussian function over the square root of $x$
\[
\frac{e^{-\pi x^2}}{\sqrt{x}}
\]
times a certain function $L(x;k)$ defined in Section \ref{L}. We use this integral formula to show that each $a_k \geq0$. We also define a sequence of polynomials $p(x;n)$ which arise naturally from the integral formula. Using a positivity property of one such polynomial $p(x;2)$ we also prove that 
\[
a_k \geq a_{k+1}\geq0.
\]

The numbers $a_k$ are elementary-symmetric functions evaluated at the zeroes of $\xi(\frac{1}{2}+i\sqrt{t})$. Therefore information about the positivity of the $a_k$ could be related to information about the zeros of $\xi(s)$. Also, the generalized Tur\'an inequalities are positivity conditions involving the Taylor series coefficients of a function. The integral formulas we present could be applied to proving the generalized Tur\'an inequalities for $\xi(s)$. That the growth of $L(x;k)$ is slow compared to the decay of Gaussian could be useful in estimating these quantities.

We now review the defining integral formula for $\xi(s)$ and the definition of $a_k$. We have for $s \in \mathbb{C}$
\[
\pi^{-\frac{s}{2}}\Gamma(\frac{s}{2})\zeta(s) = \int_{1}^\infty (\sum_{n=1}^\infty e^{-\pi n^2 x }) (u^{\frac{1-s}{2} -1}+u^{\frac{s}{2} -1})\, dx  - (\frac{1}{1-s}+\frac{1}{s}).
\]
See \cite{Stein}, chapter 6 for the derivation of this formula. Therefore we let 
\begin{equation} \label{xi integral}
\xi(s)=-s(1-s)\pi^{-\frac{s}{2}}\Gamma(\frac{s}{2})\zeta(s) = 1- s(1-s)\int_{1}^\infty (\sum_{n=1}^\infty e^{-\pi n^2 x }) (u^{\frac{1-s}{2} -1}+u^{\frac{s}{2} -1})\, dx
\end{equation}
and let $s = \frac{1}{2}+it$ to obtain
\[
\xi(\frac{1}{2}+ i t) = 1-(\frac{1}{4}+t^2)\int_{1}^\infty (\sum_{n=1}^\infty e^{-\pi n^2 x }) x^{-\frac{3}{4}} (x^{- \frac{i t}{2}}+x^{\frac{i t}{2}})\, dx.
\]
Define the coefficients $a_k$ by
\[
\xi(\frac{1}{2}+ i t)  = \sum_{k=0}^\infty (-1)^k a_kt^{2k},
\]
\[
(-1)^k \frac{(\frac{d}{dt})^{2k}}{(2k)!}\xi(\frac{1}{2}+ i t)|_{t=0}=a_k.
\]

\section{Integral formula using $L(x;k)$} \label{L}

We show how the function $L(x;k)$ arises from the integral definition of $\xi(s)$.

\begin{lemma}\label{a integral}

We have the integral formulas 
\[
a_{k+1} =4 \sum_{M=1}^\infty  \int_M^{M+1} (\sum_{n=1}^M \frac{e^{-\pi x^2}}{\sqrt{x}} (\frac{\log(\frac{n}{x})^{2k})}{(2k)!\sqrt{n}}- \frac{\log(\frac{n}{x})^{2k+2})}{(2k+2)!4\sqrt{n}})\, dx
\]
and 
\[
a_0 = \int_{-1}^1 e^{-\pi x^2}\, dx+ \sum_{M=1}^\infty\int_M^{M+1} \frac{e^{-\pi x^2 }}{\sqrt{x}} (2\sqrt{x}-\sum_{n=1}^M \frac{1}{\sqrt{n}} )\, dx.
\]
\end{lemma}
\begin{proof}
 We differentiate \eqref{xi integral} under the integral sign and use
\[
 (\frac{d}{dt})^{2k}  (x^{- \frac{i t}{2}}+x^{\frac{i t}{2}})|_{t=0} = (-1)^k2(\frac{\log(x)}{2})^{2k} 
\]
to get for integer $k \geq 1$
\begin{equation} \label{ak}
a_{k} = 2\int_1^\infty x^{-\frac{3}{4}}( \frac{(\frac{\log(x)}{2})^{2k-2}}{(2k-2)!} - \frac{(\frac{\log(x)}{2})^{2k}}{(2k)!4}) \sum_{n=1}^\infty e^{-\pi n^2 x }\, dx
\end{equation}
and 
\begin{equation} \label{a0}
a_0 = 1-2\int_1^\infty  \frac{x^{-\frac{3}{4}}}{4} \sum_{n=1}^\infty e^{-\pi n^2 x }\, dx.
\end{equation}
To the above equations for each $n$ we apply
\begin{equation} \label{each n}
\int_1^\infty x^{-\frac{3}{4}} (\frac{\log(x)}{2})^{2k}e^{-\pi n^2 x }\, dx = 2\int_n^\infty  \frac{\log(\frac{n}{x})^{2k}}{\sqrt{n}} \frac{e^{-\pi x^2}}{\sqrt{x}}\,dx
\end{equation}
where we have used the change of variables $x \mapsto \frac{x^2}{n^2}$. 
Thus for $k \geq 0$ formula \eqref{ak} becomes
\[
a_{k+1} =4 \sum_{M=1}^\infty  \int_M^{M+1} (\sum_{n=1}^M \frac{e^{-\pi x^2}}{\sqrt{x}} (\frac{\log(\frac{n}{x})^{2k})}{(2k)!\sqrt{n}}- \frac{\log(\frac{n}{x})^{2k+2})}{(2k+2)!4\sqrt{n}})\, dx
\]
and formula \eqref{a0} becomes
\begin{equation}\label{a0 becomes}
a_0 = 1-\sum_{M=1}^\infty\int_M^{M+1} \frac{e^{-\pi x^2 }}{\sqrt{x}} \sum_{n=1}^M \frac{1}{\sqrt{n}} \, dx.
\end{equation}
Now to \eqref{a0 becomes} we apply 
\[
1=\int_{-\infty}^\infty e^{-\pi x^2}\, dx
\]
and obtain 
\[
a_0 = \int_{-1}^1 e^{-\pi x^2}\, dx+ \sum_{M=1}^\infty\int_M^{M+1} \frac{e^{-\pi x^2 }}{\sqrt{x}} (2\sqrt{x}-\sum_{n=1}^M \frac{1}{\sqrt{n}}) \, dx.
\]
This completes the proof.
\end{proof}

We rewrite the formulas in Lemma \ref{a integral} as  
\[
a_{k+1} =4  \int_1^{\infty}  \frac{e^{-\pi x^2}}{\sqrt{x}} L(x;k)\, dx
\]
and 
\[
a_{0} = \int_{-1}^1 e^{-\pi x^2}\, dx+\int_1^{\infty}  \frac{e^{-\pi x^2}}{\sqrt{x}} B(x)\, dx
\]
for the functions $L(x;k)$ and $B(x)$ defined next.
\begin{definition}For $x,k \in \mathbb{R}$, define
\[
L(x;k) =  \sum_{n=1}^{\lfloor x \rfloor}  (\frac{\log(\frac{n}{x})^{2k}}{(2k)!\sqrt{n}} - \frac{\log(\frac{n}{x})^{2k+2}}{(2k+2)!4\sqrt{n}})
\]
and 
\[
B(x) = 2\sqrt{x} - \sum_{n=1}^{\lfloor x \rfloor} \frac{1}{\sqrt{n}}.
\]
\end{definition}

\begin{lemma}
For integer $k \geq1$, the function $L(x;k)$ has the bound for $x \geq e$
\[
|L(x;k)| \leq 2x\log(x)^{2k}. 
\]
For integer $k \geq 1$, $L(x;k)$ as a function of $x$ is $(2k-1)$-differentiable on $(0,\infty)$. 
\end{lemma}

\begin{proof}
The bound on $L(x;k)$ follows immediately from the definition. It is also follows from the definition that $L(x;k)$ is piecewise-smooth for positive $x$ except possibly at the integers. We prove that $L(x;k)$ is $(2k-1)$-differentiable at integer $x=M> 0$. Now the function
\[
f_1(x) = \sum_{n=1}^{M}(\frac{\log(\frac{n}{x})^{2k}}{(2k)!\sqrt{n}} - \frac{\log(\frac{n}{x})^{2k+2}}{(2k+2)!4\sqrt{x}})
\] 
is equal to $L(x;k)$ for $x \in [M,M+1)$ and the function 
\[
f_2(x) = \sum_{n=1}^{M-1}(\frac{\log(\frac{n}{x})^{2k}}{(2k)!\sqrt{n}} - \frac{\log(\frac{n}{x})^{2k+2}}{(2k+2)!4\sqrt{n}})
\]
is equal to $L(x;k)$ for $x \in [M-1,M)$.
Then the first $2k-1$ derivatives of the difference
\[
f_1(x) - f_2(x) = \frac{\log(\frac{M}{x})^{2k}}{(2k)!\sqrt{M}} - \frac{\log(\frac{M}{x})^{2k+2}}{(2k+2)!4\sqrt{M}}
\]
are each equal to 0 at $x=M$, including $f_1(M) - f_2(M)=0$. This completes the proof.
\end{proof}

The functions defined next and Lemma \ref{anti-derivative} will be used to prove the positivity of $L(x;k)$ in Theorem \ref{LB positive}. 
\begin{definition} \label{UV} For $x,k$ and $n\in \mathbb{R}$ with $x$ and $n\geq0$, we define the functions
\begin{align*}
U(x;n,k) &=\frac{\log(\frac{x}{n})^{2k}}{2k!\sqrt{n}} -\frac{\sqrt{n}}{2}(\frac{\log(\frac{x}{n})^{2k+2}}{(2k+2)!}+2\frac{\log(\frac{x}{n})^{2k+1}}{(2k+1)!})\\
V(x;n,k) &= \frac{\sqrt{n+1}}{2}(\frac{\log(\frac{x}{n+1})^{2k+2}}{(2k+2)!} +2\frac{\log(\frac{x}{n+1})^{2k+1}}{(2k+1)!}) - \frac{\log(\frac{x}{n+1})^{2k+2}}{(2k+2)!4\sqrt{n+1}}.
\end{align*}
\end{definition}

\begin{lemma}\label{anti-derivative}
For integer $k \geq0$, we have the anti-derivatives 
 \[
\int \frac{\log(\frac{z}{x})^{2k}}{(2k)!\sqrt{z}} \, dz = 2 \sqrt{z}\sum_{h=0}^{2k}\frac{(-1)^h 2^h\log(\frac{z}{x})^{2k-h}}{(2k-h)!}+C
\]
and 
\[
\int \frac{\log(\frac{z}{x})^{2k}}{(2k)!\sqrt{z}}- \frac{\log(\frac{z}{x})^{2k+2}}{(2k+2)!4\sqrt{z}} \, dz = -\frac{\sqrt{z}}{2}(\frac{\log(\frac{z}{x})^{2k+2}}{(2k+2)!} - 2 \frac{\log(\frac{z}{x})^{2k+1}}{(2k+1)!} )+C.
\]
\end{lemma}
\begin{proof}
 This follows from successive integration by parts 
\[
\int u \,dv = u v - \int v \,du
\]
using $u=\sqrt{z}$ each time. 
\end{proof}

\begin{theorem} \label{LB positive}
We have the equality
\begin{equation} \label{L as UV}
L(x;k) = V(x;0,k) + U_{\lfloor x \rfloor,k}(x)+ \sum_{n=1}^{\lfloor x \rfloor-1} V(x;n,k)+ U(x;n,k).
\end{equation}
 For $x \geq 1$ and $k \geq0$, 
\[
L(x;k) \geq0\,\,\,\,\,\,\,\text{ and} \,\,\,\,\,\,\,\, B(x) \geq0.
\]
\end{theorem}
\begin{proof}
We first prove $L(x;k) \geq0$.
Suppose $M\leq x < M+1$ where $ M\geq1$ is an integer.
By definition 
\[
L(x;k) =  \sum_{n=1}^{M}  (\frac{\log(\frac{n}{x})^{2k}}{(2k)!\sqrt{n}} - \frac{\log(\frac{n}{x})^{2k+2}}{(2k+2)!4\sqrt{n}}).
\] 
From the definitions of $U(x;n,k)$ and $V(x;n,k)$ we have for $n \geq1$
\[
U(x;n,k) +V(x;n-1,k) =  \frac{\log(\frac{n}{x})^{2k}}{(2k)!\sqrt{n}} - \frac{\log(\frac{n}{x})^{2k+2}}{(2k+2)!4\sqrt{n}}.
\]
This proves equation \eqref{L as UV}. We also have for $n \geq1$
\begin{align}
&U(x;n,k)+V(x;n,k) =  \nonumber \\ 
&\left(\frac{\log(\frac{n}{x})^{2k}}{(2k)!\sqrt{n}} -\int_n^{n+1} \frac{\log(\frac{z}{x})^{2k}}{(2k)!\sqrt{z}} \, dz \right) + 
\left(\int_{n}^{n+1} \frac{\log(\frac{z}{x})^{2k+2}}{(2k+2)!4\sqrt{z}}\, dz - \frac{\log(\frac{n+1}{x})^{2k+2}}{(2k+2)!4\sqrt{n+1}}\right).\label{U+V positive}
\end{align}
The above equation follows from Lemma \ref{anti-derivative}. 
And each expression in the parentheses of \eqref{U+V positive} is positive because the functions of $z$
\[
\frac{\log(\frac{z}{x})^{2k}}{(2k)!\sqrt{z}} \,\,\text{ and } \,\,-\frac{\log(\frac{z}{x})^{2k+2}}{(2k+2)!4\sqrt{z}}
\]
are positive decreasing and negative increasing, respectively, for $z \in (1,M)$ and $1 \leq n \leq M-1$. This shows that for $k\geq0$ and $1\leq n \leq M-1$ 
\[
U(x;n,k)+V(x;n,k)\geq 0.
\]
Now 
\[
V(x;0,k) =  \frac{\log(x)^{2k+2}}{4(2k+2)!} +\frac{\log(x)^{2k+1}}{(2k+1)!}\geq0
\]
for $x \geq1$, and we prove that $U(x;M,k)$ is positive in Lemma \ref{UM positive} for $M\leq x < M+1$. This proves the positivity of $L(x;k)$. 

For $B(x)$, we compare the integral with $M \leq x <M+1$
\[
\sum_{n=1}^M \frac{1}{\sqrt{n}}\leq \int_0^M \frac{dz}{\sqrt{z}} = 2\sqrt{M} \leq 2\sqrt{x}
\] 
which completes the proof.
\end{proof}
This is the Lemma used in Theorem \ref{LB positive}. 
\begin{lemma} \label{UM positive}For integer $M \geq1$ and $k \geq0$, and $M \leq x \leq M+1$, 

\[
U(x;M,k)=\frac{\log(\frac{M}{x})^{2k}}{(2k)!\sqrt{M}}-\frac{\sqrt{M}}{2}(\frac{\log(\frac{M}{x})^{2k+2}}{(2k+2)!} - 2 \frac{\log(\frac{M}{x})^{2k+1}}{(2k+1)!} ) \geq 0.
\]
\end{lemma}
\begin{proof}
We factor the left side of inequality in the lemma statement to obtain
\[
\frac{\log(\frac{M}{x})^{2k}}{(2k)!\sqrt{M}}(1-\frac{M}{2}(\frac{\log(\frac{M}{x})^{2}}{(2k+2)(2k+1)} - 2 \frac{\log(\frac{M}{x})}{(2k+1)} ).
\]
Now using the bound on $x$ we get 
\[
1-\frac{M}{2}(\frac{\log(\frac{M}{x})^{2}}{(2k+2)(2k+1)} - 2 \frac{\log(\frac{M}{x})}{(2k+1)}) \geq 1-\frac{M}{2}(\frac{\log(1+\frac{1}{M})^{2}}{(2k+2)(2k+1)} +2 \frac{\log(1+\frac{1}{M})}{(2k+1)} ).
\]
Now let $M =\frac{1}{y}$ where $0 \leq y \leq 1$. For $k \geq 0$ we have
\[
1-\frac{1}{2y}(\frac{\log(1+y)^{2}}{(2k+2)(2k+1)} + 2 \frac{\log(1+y)}{(2k+1)}) \geq 1-\frac{1}{2y}(\frac{\log(1+y)^{2}}{2} + 2 \log(1+y)).
\]
We then show that 
\[
2y - (\frac{\log(1+y)^{2}}{2} +2 \log(1+y))\geq 0.
\]
The above inequality is true when $y=0$, and we show that the left side is increasing in $y$. We differentiate with respect to $y$ and see that we next must show that 
\[
2-(\frac{\log(1+y)}{1+y} + \frac{2}{1+y}) \geq 0.
\]
Multiplying by $(1+y)$ gives
\[
2(1+y)-(\log(1+y) + 2) \geq 0
\]
and this is true for because for  $y\geq 0$ we have  
\[
2y \geq y \geq  \log(1+y).
\]
\end{proof}

\begin{corollary} \label{a positive}

For $k\geq0$
\[
a_{k+1}  = 4\int_1^\infty \frac{e^{-\pi x^2}}{\sqrt{x}} L(x;k) \, dx \geq0
\]
and
\[
a_0= \int_{-1}^1 e^{-\pi x^2}\, dx+\int_1^\infty \frac{e^{-\pi x^2}}{\sqrt{x}} B(x)  \, dx\geq0.
\]
\end{corollary}
\begin{proof}
These statements follow from Lemma \ref{a integral} and Theorem \ref{LB positive}.

\end{proof}

\section{The polynomials $p(x;n)$}
The following definition of the polynomials $p(x;n)$ is motivated by Lemma \ref{int by parts}.
\begin{definition}For integer $n \geq0$ define the polynomial $p(x;n)$ by
\[
p(-\pi x^2;n+1) = 2\sqrt{x}e^{\pi x^2}\frac{d}{dx}( p(-\pi x^2;n) \sqrt{x} e^{-\pi x^2}),\, \, \, \, \,
\]
\[
 p(x;0)=1.
\]
\end{definition}
The first few $p(x;n)$ are 
\begin{align*}
p(x;0)&=1\\
p(x;1)&=1+4x \\ 
p(x;2) &= 1+24x+ 16x^2\\ 
  p(x;3) &= 1+124x+240x^2+64x^3.
\end{align*}
\begin{lemma}
For integer $n\geq 0$ and $x \in \mathbb{C}$,
\begin{equation} \label{p as sum}
p(x;n) =\sum_{m=0}^\infty \frac{x^m}{m!} \sum_{j=0}^m (-1)^{j-m}{m \choose j}(4j+1)^n
\end{equation}
\begin{equation}\label{p with e}
p(x;n) = e^{-x}\sum_{j=0}^\infty (4j+1)^n \frac{x^j}{j!}
\end{equation}
\end{lemma}
\begin{proof}
We denote 
\[
p(x;n) = \sum_{m=0}^\infty c_{m,n}x^m.
\]
Thus $c_{0,0}=1$ and $c_{m,0}=0$ for $m\geq1$. The definition of $p(x;n)$ implies that
\begin{equation} \label{c relation}
c_{m,n+1} = (4m+1)c_{m,n}+4c_{m-1,n}.
\end{equation}
We prove by induction on $n$ that 
\begin{equation}\label{c formula}
c_{m,n}=\sum_{j=0}^m (-1)^{j-m}{m \choose j}(4j+1)^n.
\end{equation}
It is true for $c_{m,0}$ for all $m$. Assume it is true for $c_{m,n}$ for all $m$ for some $n \geq0$. Then we substitute \eqref{c formula} into \eqref{c relation}, compare the coefficients of $(4k+1)^n$, and apply the identity
\[
4m{m-1 \choose k} - (4m+1){m \choose k} = -(4k+1){m \choose k}
\]
for $0\leq k\leq m-1$. 
This proves the induction step. 

Now the series on the right side of \eqref{p with e} is convergent for any $x \in \mathbb{C}$. We calculate its Taylor coefficients of the whole function in \eqref{p with e} at $x=0$ and see that they match \eqref{p as sum}. This completes the proof.
\end{proof}

 \begin{lemma}\label{int by parts}
For $c>0$ and integers $m\geq0$ and $n \geq0$, we have 
\begin{align*}
&\int_c^\infty \frac{e^{-\pi x^2}}{\sqrt{x}} \frac{\log( \frac{x}{c})^{m}}{m!}\, dx =  \int_c^\infty  \frac{e^{-\pi x^2}}{\sqrt{x}} \frac{(-1)^n p(-\pi x^2;n)}{2^n} \frac{\log( \frac{x}{c})^{m+n}}{(m+n)!}\, dx\\ 
\end{align*}
\end{lemma}
\begin{proof}
This follows from integration by parts 
\[
\int u \,dv = u v - \int v \,du
\]
applied $n$ times to integral on the left side; use $\displaystyle v = \frac{\log(\frac{ x}{c})^{m+h}}{(m+h)!}$ for the $h$-th application of integration by parts. 
\end{proof}

 \begin{theorem} \label{ak with p}
 For integers $k\geq0$ and $n \geq0$, 
 \[
 a_{k+1} = 4\int_1^\infty \frac{e^{-\pi x^2}}{\sqrt{x}}  \frac{(-1)^np(-\pi x^2;n)}{2^n}L_{k+\frac{n}{2}}(x) \, dx.
 \]
 \end{theorem}
 \begin{proof}
 Recall
 \[
a_{k+1} = 4\sum_{M=1}^\infty \int_M^{M+1} \frac{e^{-\pi x^2} }{\sqrt{x}}  \sum_{n=1}^{M}  (\frac{\log(\frac{n}{x})^{2k}}{(2k)!\sqrt{n}} - \frac{\log(\frac{n}{x})^{2k+2}}{(2k+2)!4\sqrt{n}})\, dx. 
 \]
We re-arrange this as 
\[
a_{k+1} = 4\sum_{M=1}^\infty \int_{M}^\infty \frac{e^{-\pi x^2} }{\sqrt{x}}(\frac{\log(\frac{M}{x})^{2k}}{(2k)!\sqrt{M}} - \frac{\log(\frac{M}{x})^{2k+2}}{(2k+2)!4\sqrt{M}})\, dx.
\]
 We apply Lemma \ref{int by parts} to get 
\[
a_{k+1} = 4\sum_{M=1}^\infty \int_{M}^\infty \frac{e^{-\pi x^2} }{\sqrt{x}} \frac{(-1)^np(-\pi x^2;n)}{2^n}(\frac{\log(\frac{M}{x})^{2k+n}}{(2k+n)!\sqrt{M}} - \frac{\log(\frac{M}{x})^{2k+2+n}}{(2k+2+n)!4\sqrt{M}})\, dx.
\]
We re-arrange in the reverse manner to obtain 
\begin{align*}
a_{k+1} &= 4\sum_{M=1}^\infty \int_{1}^\infty \frac{e^{-\pi x^2} }{\sqrt{x}}\frac{(-1)^np(-\pi x^2;n)}{2^n}(\frac{\log(\frac{M}{x})^{2k+n}}{(2k+n)!\sqrt{M}} - \frac{\log(\frac{M}{x})^{2k+2+n}}{(2k+2+n)!4\sqrt{M}})\, dx \\ 
& = 4 \int_1^\infty \frac{e^{-\pi x^2} }{\sqrt{x}}\frac{(-1)^np(-\pi x^2;n)}{2^n}L_{k+\frac{n}{2}}(x)\, dx.
\end{align*}

 \end{proof}

We define $\mathrm{Wallis}(N)$ for estimates in the next results.
\begin{definition}
The Wallis product \emph{\cite{Wallis}} is 
\[
\frac{\pi}{2} =\prod_{n=1}^\infty \frac{(2n)^2}{(2n-1)(2n+1)}
\]
with
\[
\frac{(2n)^2}{(2n-1)(2n+1)} \geq 1
\]
for $n \geq1$.
We denote
\[
\mathrm{Wallis}(N) = \prod_{n=1}^N \frac{(2n)^2}{(2n-1)(2n+1)}
\]
so 
\[
\frac{\pi}{2}  \geq \mathrm{Wallis}(N)
\]
for all $N\geq1$.
\end{definition}

\begin{theorem} \label{a decreasing}
For integer $k\geq0$,
\[
a_k \geq a_{k+1}\geq0.
\]
\end{theorem}
\begin{proof}

For $k\geq0$, we use Theorem \ref{ak with p} with $n=2$ to obtain
\[
a_{k+1} -a_{k+2} =  \int_1^\infty \frac{e^{-\pi x^2}}{\sqrt{x}} (\frac{p(-\pi x^2;2)}{4}-1) L_{k+1}(x)\, dx
\]
Since $\displaystyle \frac{p(-\pi x^2;2)}{4}-1\geq 0$ by Lemma \ref{p2} and $L_{k+1}(x)\geq 0$ by Lemma \ref{LB positive} for $x \geq 1$, this completes the proof for $k \geq 1$.

We next prove $a_0 >a_1$. Recall 
\begin{equation}\label{recall a0}
a_0 = 2 \int_0^\infty e^{-\pi x^2} \, dx- \int_1^\infty \frac{e^{-\pi x^2}}{\sqrt{x}}\sum_{n=1}^{\lfloor x \rfloor} \frac{1}{\sqrt{n}}\, dx.
\end{equation}
and 
\begin{equation}\label{a1}
a_1 = 4\int_1^\infty \frac{e^{-\pi x^2}}{\sqrt{x}}\sum_{n=1}^{\lfloor x \rfloor} \frac{1}{\sqrt{n}} - \frac{\log(x)^2}{2!4\sqrt{n}} \, dx . 
\end{equation}
 We apply the change of variables $x \mapsto x-1$ to the left integral of \eqref{recall a0} and obtain
 \[
 a_0 =  \int_1^\infty \frac{e^{-\pi x^2}}{\sqrt{x}} (e^{2\pi x - \pi}2\sqrt{x}\, dx- \int_1^\infty \frac{e^{-\pi x^2}}{\sqrt{x}} \sum_{n=1}^{\lfloor x \rfloor} \frac{1}{\sqrt{n}})\, dx.
  \]
  Thus 
 
  \[
  a_0-a_1\geq \int_1^\infty \frac{e^{-\pi x^2}}{\sqrt{x}}\left((e^{2\pi x - \pi}-5)2\sqrt{x}+ 5(2\sqrt{x} - \sum_{n=1}^{\lfloor x \rfloor} \frac{1}{\sqrt{n}}) \right)\, dx.
  \]
  We have seen in Theorem \ref{a positive} that 
\[
2\sqrt{x} - \sum_{n=1}^{\lfloor x \rfloor} \frac{1}{\sqrt{n}} \geq 0.
\]
And for $x \geq 1$
\begin{align*}
e^{2\pi x - \pi}-5 &\geq e^{2(\frac{\pi}{2})} - 5\\ 
&\geq 1 + 2\mathrm{Wallis}(1)+\frac{2^2}{2}\mathrm{Wallis}(1)^2-5\\ 
&=\frac{55}{9}-5\\ 
&\geq 0. 
\end{align*}
This completes the proof.
 \end{proof}
 This is the lemma used in Theorem \ref{a decreasing}.
\begin{lemma} \label{p2}
For $x \geq 1$,
\[
\frac{p(-\pi x^2;2)}{4}-1=  \frac{(1-24 \pi  x^2+ 16 \pi^2 x^4)}{4}-1\geq0.
\]
\end{lemma}
\begin{proof}
We have 
\[
\frac{p(-\pi x^2;2)}{4}-1 = (-\frac{3}{4} -6 \pi+4 \pi^2) + (-6\pi+8\pi^2)(x^2-1)+ 4\pi^2(x^2-1)^2.
\]

Then 
\begin{align*}
(-6\pi+8\pi^2) = 2\pi(8 (\frac{\pi}{2})-3)\geq0
\end{align*}
where we bound $\frac{\pi}{2}$ from below with $\mathrm{Wallis}(0)=1$. 

And 
\begin{align*}
(-\frac{3}{4} -6 \pi+4 \pi^2)& = 4(\frac{\pi}{2})(4(\frac{\pi}{2})-3)-\frac{3}{4}\\
&\geq 4(\frac{\pi}{2})(4\cdot 1-3)-\frac{3}{4}\\
&\geq 4\cdot1-\frac{3}{4}\\ 
&\geq 0.
\end{align*}
where we have again bounded $\frac{\pi}{2}$ with $\mathrm{Wallis}(0)=1$.
This proves the lemma.
\end{proof}

\section{Further Work}

\begin{itemize}

\item Apply the integral formulas to proving 
\[
\sum_{j=0}^n(-1)^j {n \choose j} a_{k+j}\geq0
\] 
 and the generalized Tur\'an inequalities. 
 
 \item Relate the polynomials $p(x;n)$ to known polynomial families.
 
 \item See if similar results hold for Dirichlet $L$-functions.
 
\end{itemize}

\end{document}